\documentclass[A4,11pt,oneside]{amsart}
\usepackage{amsthm,amsmath,amssymb,eucal}
\usepackage{geometry}
\usepackage{resizegather}
\usepackage{graphicx}
\usepackage{color}
\usepackage[normalem]{ulem}
\usepackage[latin1]{inputenc}
\usepackage{mathtools}
\usepackage{url}
\usepackage{caption}
\usepackage{subcaption}
\mathtoolsset{showonlyrefs}

\newtheorem{claim}{Claim}[section]
\newtheorem{theorem}[claim]{Theorem}

\newtheorem{lemma}[claim]{Lemma}

\theoremstyle{definition}

\newcommand{\soutg}{\bgroup\markoverwith{\textcolor{green}{\rule[.5ex]{2pt}{1pt}}}\ULon}
\newcommand{\soutb}{\bgroup\markoverwith{\textcolor{blue}{\rule[.5ex]{2pt}{1pt}}}\ULon}
\newcommand{\soutr}{\bgroup\markoverwith{\textcolor{red}{\rule[.5ex]{2pt}{1pt}}}\ULon}

\DeclareMathOperator*{\re}{Re}

\usepackage[normalem]{ulem}
\definecolor{DarkGreen}{rgb}{0,0.5,0.1} 
\definecolor{DarkBlue}{rgb}{0,0.1,0.5} %

\newcommand\soutD{\bgroup\markoverwith
{\textcolor{DarkGreen}{\rule[.5ex]{2pt}{1pt}}}\ULon}
\newcommand\soutP{\bgroup\markoverwith
{\textcolor{blue}{\rule[.5ex]{2pt}{1pt}}}\ULon}
\newcommand{\Hm}[1]{\leavevmode{\marginpar{\tiny%
$\hbox to 0mm{\hspace*{-0.5mm}$\leftarrow$\hss}%
\vcenter{\vrule depth 0.1mm height 0.1mm width \the\marginparwidth}%
\hbox to
0mm{\hss$\rightarrow$\hspace*{-0.5mm}}$\\\relax\raggedright #1}}}

\title[Spectral determinant for the wave equation with Dirac damping]{Spectral determinant for the wave equation on an interval with Dirac damping}

\author{David Krej\v{c}i\v{r}\'ik} 
\address{Department of Mathematics,
	Faculty of Nuclear Sciences and Physical Engineering,
	Czech Technical University in Prague,
	Trojanova 13, 12000 Prague, Czech Republic}
\email{david.krejcirik@fjfi.cvut.cz}

\author{Ji\v{r}\'{\i} Lipovsk\'{y}}
\address{Department of Physics, Faculty of Science, University of Hradec Kr\'alov\'e, Rokitansk\'eho 62,
500\,03 Hradec Kr\'alov\'e, Czechia}
\email{jiri.lipovsky@uhk.cz}

\date{April 16, 2024}

\begin{document}

\begin{abstract}
A closed formula for the spectral determinant for the wave equation
on a bounded interval, subject to Dirichlet boundary conditions
and an $\alpha$-multiple of the Dirac $\delta$-type damping,
is derived.
Depending on the choice of the branch cut of the logarithm used in its definition, the spectral determinant diverges either for $\alpha =2$ or $\alpha=-2$.
\end{abstract}
\maketitle

\section{Introduction}
A string instrument is classically modelled by the wave equation
on an interval $(0, L)$ of length $L>0$, 
subject to Dirichlet boundary conditions.
In 1982 Bamberger, Rauch, and Taylor \cite{BRT82} suggested 
explaining the playing of harmonics by modelling  
the finger pressure at a point $a \in (0, L)$ 
by a strongly localised frictional resistance around~$a$. 
More specifically, their mathematical model is 
the damped wave equation 
\begin{equation}\label{wave}
   \partial_{tt}u + \alpha \delta_a \partial_t u - \partial_{xx} u = 0\,,
\end{equation}
in space-time variables $(x,t) \in (0, L) \times (0,\infty)$,
where $\delta_a$ is the Dirac delta distribution centered at~$a$
and~$\alpha$ is a positive number.
In this paper, 
motivated by relativistic quantum mechanics \cite[Sec.~7]{KR23},
we allow~$\alpha$ to be an arbitrary complex number. 

It turns out that there is an abrupt change in time-evolution
properties of~\eqref{wave}
when the complex parameter~$\alpha$ takes values~$\pm 2$.
In fact, $\alpha=2$ is the optimal friction of~\cite{BRT82}.
To see it, the strategy of~\cite{BRT82} is to make a careful 
spectral analysis of the associated damped wave operator
(to be properly introduced in Section~\ref{sec:model})
\begin{equation}\label{damped}
  H_\alpha =
  \begin{pmatrix}
  -\alpha \delta_a &  \frac{\partial^2}{\partial x^2}
  \\ I & 0
  \end{pmatrix}\,.
\end{equation}

This spectral analysis was continued by Cox and Henrot in~\cite{CH08} 
by studying basis properties of the eigenfunctions. 
In particular, they show that the root vectors form the Riesz basis
if $\alpha \not=\pm 2$.
While the opposite implication is not established in~\cite{CH08},
it is true that the root vectors do not form the Riesz basis anymore
if $\alpha =\pm 2$.
In fact, some of the eigenvalues of this operator diverge to complex infinity 
when $\alpha$ approaches the two critical values~$\pm 2$,
losing thus even completeness.

Wild spectral properties of the damped wave operator
for $\alpha = \pm 2$
appear on unbounded geometries too, 
see \cite[Rem.~1]{KK8} and~\cite{KR23}.
What is more, the ``magical'' value~$2$ is explained in~\cite{KR23}
by considering the damped wave equation on non-compact star graphs.
It turns out, that the abrupt change in spectral properties happens precisely
at~$\pm \alpha$ equal the number of edges emanating from the vertex
where the distributional damping is placed.

The objective of this paper is to quantify the transition at 
$\alpha = \pm 2$ by considering yet another spectral quantity 
-- the spectral determinant.
This generalisation of the notion of the determinant of a matrix 
for possibly unbounded operators 
was introduced by Ray and Singer in 1971 \cite{RS71}. 
The spectral determinants were particularly
studied for the Sturm--Liouville operators \cite{LS77, GK19}, Laplacians in various domains \cite{AS94}, harmonic and anharmonic oscillators \cite{Fre18} or used in string theory or quantum field theory, 
see, \cite{Dun08} and references therein.
The present paper points out another class of operators
for which the spectral determinant can be computed in a closed form.

To avoid the fact that the product of eigenvalues~$\lambda_j$ 
of an operator~$H$ with compact resolvent  
may not be convergent, the spectral determinant of~$H$
is defined through the spectral zeta function 
\begin{equation}\label{zeta}
  \zeta(s) = \sum_{j=1}^\infty \lambda_j^{-s}
  \qquad \mbox{with} \qquad
  \lambda_j^{-s} = \mathrm{e}^{-s \log{\lambda_j}}
  \,.
\end{equation}
This sum is usually convergent in the half-plane $\mathrm{Re}\,s>s_0>0$. To give a proper meaning to the derivative of the zeta function at zero, the zeta function is meromorphically continued from the set where the sum converges to the rest of the complex plane in $s$. Since this continuation is unique, the value of $\zeta'(0)$ is properly defined and the determinant can be introduced as 
\begin{equation}
 \mathrm{det}(H) = \mathrm{exp}\,(-\zeta'(0))\,.\label{eq:det}
\end{equation}

However, the power of the eigenvalues in the definition of the zeta function~\eqref{zeta} 
depends on how we introduce the complex natural logarithm. One can choose different branches of the logarithm and the branch cut determines from which interval the arguments of the eigenvalues are chosen. Consequently, the spectral determinant may differ for different branch cuts, see \cite{QHS93, FL19, LM23}. This behaviour appears also for the problem considered in this paper. 

When we move the branch cut so that it crosses a finite number of eigenvalues, the spectral determinant does not change (see \cite{FL19, LM23}). As the number of eigenvalues of~$H_\alpha$ with the arguments 
in the intervals $(\beta_1,\beta_2)$ 
with $\frac{\pi}{2}<\beta_1<\beta_2<\frac{3\pi}{2}$
and $(\beta_3,\beta_4)$ with $-\frac{\pi}{2}<\beta_3<\beta_4<\frac{\pi}{2}$ is finite, one essentially has two independent choices of the branch cut. Namely, anywhere in the sectors with the argument in the intervals $(\beta_1,\beta_2)$ and $(\beta_3,\beta_4)$, respectively, as long as the cut does not intersect any eigenvalue (otherwise the spectral zeta function would not be properly defined).
To be more specific, we choose the first and second branch cuts 
as the rays $\mathrm{e}^{i(\pi+\varepsilon)} (0,\infty)$
and $\mathrm{e}^{-i\varepsilon} (0,\infty)$, respectively,
where $\varepsilon \in (0,\frac{\pi}{2})$ is always such that there are no eigenvalues of~$H_\alpha$ with arguments in the intervals $(\pi, \pi+\varepsilon]$ and $[-\varepsilon,0)$ (in principle, $\varepsilon$ can depend on~$\alpha$).
To make the long story short, 
we concisely say that the branch cut is either
``below the negative real axis'' or ``below the positive real axis'',
respectively.

Writing $\mathrm{det}(\alpha) = \mathrm{det}(H_\alpha)$ 
in the present case,
our main result reads as follows.
\begin{theorem}\label{thm:main}
If the branch cut of the logarithm is below the negative real axis, then
$$
  \mathrm{det}(\alpha) =
\begin{cases}
  \displaystyle \frac{4L}{2-\alpha}
  & \mbox{if} \quad \alpha\ne 2 \mbox{ and } a \in (0,L)
  \,, \\ 
  2L 
  & \mbox{if} \quad \alpha=2 \mbox{ and } a \not= L/2
  \,, \\ 
  L 
  & \mbox{if} \quad \alpha=2 \mbox{ and } a = L/2
  \,. \\ 
\end{cases}
$$
If the branch cut of the logarithm is below the positive real axis, then
$$
  \mathrm{det}(\alpha) =
\begin{cases}
  \displaystyle - \frac{4L}{2+\alpha}
  & \mbox{if} \quad \alpha\ne -2 \mbox{ and } a \in (0,L)
  \,, \\ 
  -2L 
  & \mbox{if} \quad \alpha=-2 \mbox{ and } a \not= L/2
  \,, \\ 
  -L 
  & \mbox{if} \quad \alpha=-2 \mbox{ and } a = L/2
  \,. \\ 
\end{cases}
$$
\end{theorem}

The results of the theorem are illustrated in Figure~\ref{fig:det}.

\begin{figure}[h]
\centering
\begin{subfigure}{.47\textwidth}
  \centering\captionsetup{width=.9\linewidth}
  \includegraphics[width=.9\linewidth]{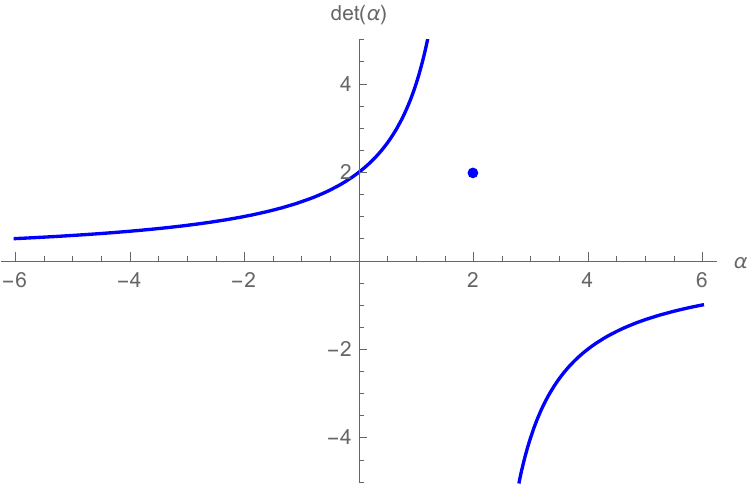}
  \caption{The branch cut is below the negative real axis, 
  $L=1$, $a\ne \frac{1}{2}$.}
  \label{fig1}
\end{subfigure}%
\hspace{0.05\textwidth}
\begin{subfigure}{.47\textwidth}
  \centering\captionsetup{width=.9\linewidth}
  \includegraphics[width=.9\linewidth]{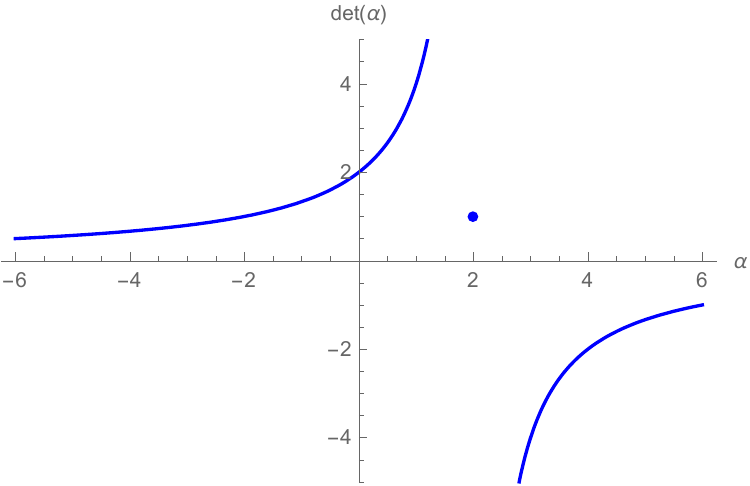}
  \caption{The branch cut is below the negative real axis, 
  $L=1$, $a = \frac{1}{2}$.}
  \label{fig2}
\end{subfigure}%
\\
\begin{subfigure}{.47\textwidth}
  \centering\captionsetup{width=.9\linewidth}
  \includegraphics[width=.9\linewidth]{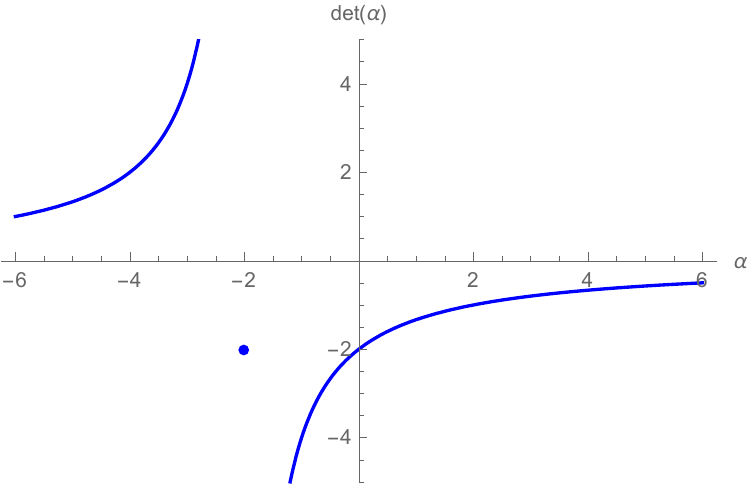}
  \caption{The branch cut is below the positive real axis, 
  $L=1$, $a\ne \frac{1}{2}$.}
  \label{fig3}
\end{subfigure}%
\hspace{0.05\textwidth}
\begin{subfigure}{.47\textwidth}
  \centering\captionsetup{width=.9\linewidth}
  \includegraphics[width=.9\linewidth]{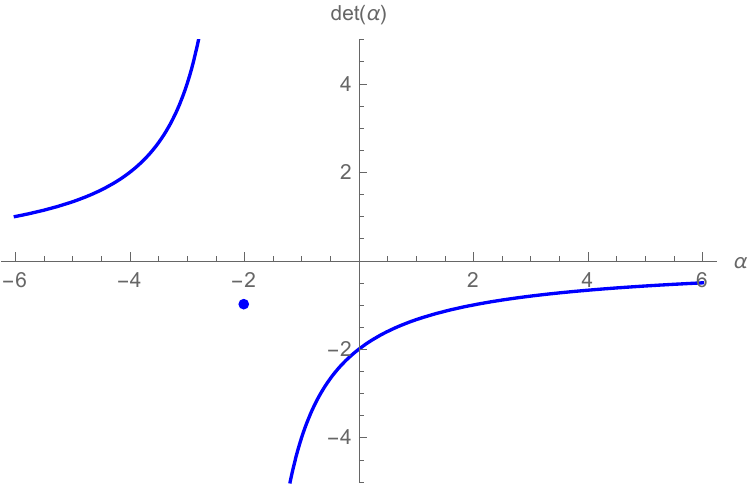}
  \caption{The branch cut is below the positive real axis, 
  $L=1$, $a = \frac{1}{2}$.}
  \label{fig4}
\end{subfigure}%
\caption{Dependence of the spectral determinant on the real values of the parameter $\alpha$ for both branch cuts and different positions of the $\delta$-damping.}
\label{fig:det}
\end{figure}

The paper is structured as follows. 
In Section~\ref{sec:model} we rigorously introduce the model. 
Sections~\ref{sec:3} and~\ref{sec:4} are devoted to the proof of Theorem~\ref{thm:main}. 
The strategy is to consider the case of rational ratio $a/L$ first 
(Section~\ref{sec:3}) and then argue by continuity (Section~\ref{sec:4}).

\section{The model}\label{sec:model}
Recall that
the wave equation~\eqref{wave} is considered for any complex number~$\alpha$
and subject to the Dirichlet boundary conditions $u(0,t) = u(L,t) = 0$
for all times $t > 0$. 
In order to reconsider it as an evolution problem for a linear operator,
we follow~\cite{BRT82} and introduce the pair 
$U = (v,w)^\mathrm{T} = (\frac{\partial u}{\partial t},u)^\mathrm{T	}$ 
in the Hilbert space 
$$
  \mathcal{H} = L^2((0,L))\oplus H^1((0,L))
  \,.
$$ 
Then the evolution problem~\eqref{wave} 
is equivalent to 
$$
  \frac{\partial U}{\partial t} = H_\alpha U
$$ 
with the operator~$H_\alpha$ formally acting as in~\eqref{damped} in~$\mathcal{H}$. 
More specifically, the distributional damping
as well as the Dirichlet boundary conditions 
are introduced via the operator domain.
Thus the rigorous definition of~$H_\alpha$ reads
$$
\begin{aligned}
  H_\alpha \begin{pmatrix}v\\w\end{pmatrix} (x)
  &= \begin{pmatrix} w''(x)\\v(x)\end{pmatrix}
  \quad \mbox{for every} \quad
  x \in (0,a) \cup (a,L) \,,
  \\
   D(H_\alpha) 
   &= \left\{
   \left(
   \begin{aligned}
   v &\in H_0^1((0,L))
   \\
   w &\in H_0^1((0,L))\cap H^2((0,a))\cap H^2((a,L))
   \end{aligned}
   \right)
   : \
   w'(a^+)-w'(a^-) = \alpha v(a)\right\} .
\end{aligned}
$$
The operator~$H_\alpha$ is maximal dissipative 
(respectively, maximal accretive) 
if $\re \alpha \geq 0$ 
(respectively, $\re\alpha \leq 0$).
In particular, $H_\alpha$~is skew-adjoint if $\re\alpha=0$.
Moreover, $H_\alpha$~is an operator with compact resolvent.
We refer to \cite{BRT82,CH08, KR23} for more details. 
An alternative approach based on Schur complement
can be found in \cite{Ger24}.

Since the domain $D(H_\alpha)$ is compactly embedded in~$\mathcal{H}$,
the resolvent of~$H_\alpha$ is compact. 
Consequently, the spectrum of~$H_\alpha$ is purely discrete.
We arrange the eigenvalues in a sequence $\{\lambda_j\}_{j=1}^\infty$,
where each eigenvalue is repeated according to its multiplicity. 
It is easy to see  (see, e.g., \cite{CH08}) that
the eigenvalues of~$H_\alpha$ satisfy the implicit equation 
\begin{equation}
  \sinh{(L\lambda)}+\alpha \sinh{(a\lambda)}\sinh{((L-a)\lambda)} = 0\,.\label{eq:basic}
\end{equation}
An analysis of this equation implies that
$\lambda_j$~is asymptotically proportional to~$j$ as $j \to \infty$.
It follows that the sum in the definition of the spectral zeta function~\eqref{zeta} 
is convergent in the half-plane $\re s > 1$. 
For those $s$, for which the above sum is not convergent, we define its values as the unique meromorphic continuation of the spectral zeta function from the region where it converges. The spectral determinant of~$H_\alpha$
is then defined by formula~\eqref{eq:det}.

As explained above Theorem~\ref{thm:main},
we consider two branches of the complex logarithm in~\eqref{zeta}:
either the ray $\mathrm{e}^{i(\pi+\varepsilon)} (0,\infty)$
or $\mathrm{e}^{-i\varepsilon} (0,\infty)$.  
Any other eligible choice of the branch cut 
is already covered by Theorem~\ref{thm:main}.

\section{Rational position of the Dirac damping condition}\label{sec:3}
Our first step in proving Theorem~\ref{thm:main}
is the analysis of the case when the $\delta$-damping divides 
the interval $(0,L)$ into two parts with rationally related lengths. 
Let us therefore consider the situation when 
$$
  a = p L_0
  \qquad \mbox{and} \qquad
  L-a = q L_0
$$ 
with $p$ and $q$ being incommensurable positive integers, 
without loss of generality $p\geq q$,
and~$L_0$ is a positive number actually defined by the requirement. 
Then the equation \eqref{eq:basic} can be rewritten into the form 
$$
  (2+\alpha)\mathrm{e}^{2(p+q)L_0\lambda}-\alpha \mathrm{e}^{2 p L_0 \lambda}-\alpha \mathrm{e}^{2 q L_0 \lambda}+(\alpha-2) = 0\,.
$$
This is a polynomial equation in the variable $z = \mathrm{e}^{2 L_0\lambda}$ with one of the roots equal to 1. 
Factoring out the $z-1$ term one finds
\begin{align}
   (\alpha+2)z^{p+q-1}+(\alpha+2)z^{p+q-2}+\dots (\alpha+2)z^p+2z^{p-1}+2z^{p-2}+\dots + 2z^q &\\
  +(2-\alpha)z^{q-1}+(2-\alpha)z^{q-2}+\dots+2-\alpha &= 0\,.\label{eq:z3}
\end{align}
For $\alpha \ne -2$ we obtain, after dividing by $\alpha+2$,
\begin{align}
   z^{p+q-1}+z^{p+q-2}+\dots z^p+\frac{2}{\alpha+2}z^{p-1}+\frac{2}{\alpha+2}z^{p-2}+\dots + \frac{2}{\alpha+2}z^q &\\
  +\frac{2-\alpha}{\alpha+2}z^{q-1}+\frac{2-\alpha}{\alpha+2}z^{q-2}+\dots+\frac{2-\alpha}{\alpha+2} &= 0\,.\label{eq:z}
\end{align}
The left-hand side of the last equation can be rewritten as
\begin{equation}
  \prod_{k =1}^{p+q-1}(z-z_k)\,,\label{eq:z2}
\end{equation}
where $z_k$, $k = 1,\dots, p+q-1$ are the roots of \eqref{eq:z}. Clearly, one finds, by comparing the left-hand side of \eqref{eq:z} and \eqref{eq:z2} after substituting $z=1$, that
\begin{equation}
  \prod_{k =1}^{p+q-1}(1-\mathrm{e}^{2 L_0\mu_k}) = q+\frac{2}{\alpha+2}(p-q)+\frac{2-\alpha}{\alpha+2}q = \frac{2(p+q)}{\alpha+2}\,,\label{eq:prod1}
\end{equation}
where $\mu_k =\frac{1}{2L_0} \log{z_k}$ with $\log$ denoting the natural logarithm. 
Although we can use any $\mu_k$ satisfying the defining relation in our construction, for unambiguity of the definition we use the solutions with $\mathrm{Im\,}\mu_k\in (-\frac{\pi}{2L_0},\frac{\pi}{2L_0}]$. 
 
The equation \eqref{eq:z3} can be for $\alpha \ne 2$ written also as the equation for the variable $y =  \mathrm{e}^{-2 L_0\lambda}$ as
\begin{align}
   y^{p+q-1}+y^{p+q-2}+\dots y^p+\frac{2}{2-\alpha}y^{p-1}+\frac{2}{2-\alpha}y^{p-2}+\dots + \frac{2}{2-\alpha}y^q &\\
  +\frac{\alpha+2}{2-\alpha}y^{q-1}+\frac{\alpha+2}{2-\alpha}y^{q-2}+\dots+\frac{\alpha+2}{2-\alpha}&= 0\,.\label{eq:y}
\end{align} 
Rewriting the left-hand side as $\prod_{k =1}^{p+q-1}(y-y_k)$ with $y_k$, $k = 1,\dots, p+q-1$ being its roots and again substituting $y=1$ one obtains from \eqref{eq:y}
\begin{equation}
  \prod_{k =1}^{p+q-1}(1-\mathrm{e}^{-2 L_0\mu_k}) = q+\frac{2}{2-\alpha}(p-q)+\frac{\alpha+2}{2-\alpha}q = \frac{2(p+q)}{2-\alpha}\,.\label{eq:prod2}
\end{equation}
We will use the formul\ae\ \eqref{eq:prod1} and \eqref{eq:prod2} later.

The first set of eigenvalues of the considered problem is $\lambda_{1j} = \frac{j\pi i}{L_0}$, $j\in\mathbb{Z}\backslash \{0\}$. This set follows from the equation $z=1$. Zero is not an eigenvalue because the corresponding eigenfunction would be identically zero. The equation \eqref{eq:z} yields the eigenvalues
$$
  \lambda_{2kj} = \mu_k+\frac{\pi i}{L_0}j \,,\quad j\in \mathbb{Z}\,.
$$
Note that $\lambda_{2kj}$ for some $k$ diverge to complex infinity as $\alpha \to\pm 2$. Also note that it follows from~\eqref{eq:prod1} that $\mu_k \ne 0$ and hence all $\lambda_{2kj}$ are present in the spectrum.

\subsection{First branch cut}
Let us start to study the spectral determinant for 
the branch cut of the logarithm to be the ray 
$\mathrm{e}^{i(\pi+\varepsilon)} (0,\infty)$.

The spectral zeta function is 
$$
 \zeta(s) = \sum_{j\in\mathbb{Z}\backslash\{0\}}\lambda_{1j}^{-s}+\sum_{k=1}^{p+q-1}\sum_{j\in\mathbb{Z}}\lambda_{2kj}^{-s}\,.
$$
Note that none of $\mu_k$'s is zero due to \eqref{eq:prod1} and hence the sum in the second term goes through~$\mathbb{Z}$. The equation above can be rewritten as
\begin{align}
  \zeta(s) = & \sum_{j=1}^{\infty} \left(\frac{\pi}{L_0}\right)^{-s} \mathrm{e}^{-i\frac{\pi}{2}s} j^{-s}+\sum_{j=1}^{\infty} \left(\frac{\pi}{L_0}\right)^{-s} \mathrm{e}^{i\frac{\pi}{2}s} j^{-s} + \sum_{k=1}^{p+q-1}\mu_k^{-s} 
\\ 
& +\sum_{k=1}^{p+q-1}\left(\sum_{j=1}^{\infty} \left(\mu_k +\frac{\pi}{L_0}i j\right)^{-s} +\sum_{j=1}^{\infty} \left(\mu_k -\frac{\pi}{L_0}i j\right)^{-s} \right)\,.\label{eq:ratzeta0}
\end{align}

Defining the Riemann zeta function as $\zeta_\mathrm{R} = \sum_{j=1}^\infty j^{-s}$ for complex $s$ and Hurwitz zeta function as $\zeta_\mathrm{H}(s,c) = \sum_{j=0}^{\infty} (j+c)^{-s}$ for complex parameters $s$ and $c$ we can rewrite the previous expression as
\begin{align}
  \zeta(s) & = 2\cos{\frac{\pi s}{2}}\mathrm{e}^{s\log{\frac{L_0}{\pi}}}\zeta_\mathrm{R}(s) + \sum_{k=1}^{p+q-1}\mathrm{e}^{-s \log{\mu_k}}+\sum_{k=1}^{p+q-1}\left[\left(\frac{\pi}{L_0}\right)^{-s}\mathrm{e}^{-\frac{i\pi}{2}s}\zeta_\mathrm{H}(s,1-i\frac{\mu_k L_0}{\pi})\right.
\\
  & \hspace*{5mm} \left.+\left(\frac{\pi}{L_0}\right)^{-s}\mathrm{e}^{\frac{i\pi}{2}s}\zeta_\mathrm{H}(s,1+i\frac{\mu_k L_0}{\pi})\right]
\\
  & = 2\cos{\frac{\pi s}{2}}\mathrm{e}^{s\log{\frac{L_0}{\pi}}}\zeta_\mathrm{R}(s) +\sum_{k=1}^{p+q-1} \left[\mathrm{e}^{-s \log{\mu_k}}+\mathrm{e}^{s\log{\frac{L_0}{\pi}}}\mathrm{e}^{-\frac{i\pi}{2}s}\zeta_\mathrm{H}(s,1-i\frac{\mu_k L_0}{\pi})\right.
\\
&\hspace*{5mm} \left.+\mathrm{e}^{s\log{\frac{L_0}{\pi}}}\mathrm{e}^{\frac{i\pi}{2}s}\zeta_\mathrm{H}(s,1+i\frac{\mu_k L_0}{\pi})\right].
\label{eq:ratzeta1}
\end{align}
Note that the factor of 1 in the second argument of the Hurwitz zeta function appears because in its definition the sum goes from 0, not from 1 as \eqref{eq:ratzeta0}.

The right-hand side converges for $\re s>1$ and can be uniquely meromorphically extended to $s=0$. Our aim is to compute its derivative with respect to $s$ at $s=0$. To do that, we will use the following expressions (see \cite[eqs. 64:10:4 and 64:3:2]{SO87})
\begin{align}
  \zeta_\mathrm{R}(0) & = -\frac{1}{2}\,,\label{eq:riemann1}\\ 
  \zeta_\mathrm{R}'(0) & = -\frac{1}{2}\log{(2\pi)}\,,\label{eq:riemann2}\\
  \zeta_\mathrm{H}(0,1+c) & = -\frac{1}{2}-c\,,\label{eq:hurwitz1}\\
  \frac{\partial}{\partial s}\zeta_\mathrm{H}(0,1+c) & = \log\left(\frac{\Gamma(1+c)}{\sqrt{2\pi}}\right)\label{eq:hurwitz2}\,.  
\end{align}
Using the above formulae, 
one can differentiate the expression \eqref{eq:ratzeta1} at $s=0$ and find
\begin{align*}
  \zeta'(0) &= 2 \log{\frac{L_0}{\pi}}\zeta_\mathrm{R}(0)+ 2 \zeta_\mathrm{R}'(0)+\sum_{k=1}^{p+q-1}\left[-\log{\mu_k}+\left(\log{\frac{L_0}{\pi}}-\frac{i\pi}{2}\right)\zeta_\mathrm{H}(0,1-i\frac{\mu_k L_0}{\pi})\right.\\ 
& \hspace*{5mm} + \left.\left(\log{\frac{L_0}{\pi}}+\frac{i\pi}{2}\right)\zeta_\mathrm{H}(0,1+i\frac{\mu_k L_0}{\pi})+\frac{\partial}{\partial s}\zeta_\mathrm{H}(0,1-i\frac{\mu_k L_0}{\pi})+\frac{\partial}{\partial s}\zeta_\mathrm{H}(0,1+i\frac{\mu_k L_0}{\pi})\right]\\
 & = -\log{\frac{L_0}{\pi}}-\log{(2\pi)}+\sum_{k=1}^{p+q-1}\left[-\log{\mu_k}-\log{\frac{L_0}{\pi}}+\mu_k L_0\right.
\\
& \hspace*{5mm} \left.+ \log{\left(\frac{\Gamma(1-i\frac{\mu_k L_0}{\pi})\Gamma(1+i\frac{\mu_k L_0}{\pi})}{2\pi}\right)}\right]\,.
\end{align*}

Using (see, e.g., \cite[43:11:2]{SO87})
\begin{equation}
  \Gamma(1-ic)\Gamma(1+ic) = \frac{c\pi}{\sinh{c\pi}}\,, 
\end{equation}
we obtain
$$
  \zeta'(0) = -(p+q) \log{2} -\log{L_0}+\sum_{k=1}^{p+q-1} [L_0\mu_k-\log{(\sinh{(L_0\mu_k)})}]\,.
$$
We find
\begin{equation}
  \mathrm{det} (\alpha)= \mathrm{e}^{-\zeta'(0)} = 2^{p+q} L_0 \prod_{k=1}^{p+q-1}[\sinh{(L_0\mu_k)}\,\mathrm{e}^{-L_0\mu_k}]  = 2L_0 \prod_{k=1}^{p+q-1}(1- \mathrm{e}^{-2 L_0 \mu_k})\,.
\end{equation}
Using \eqref{eq:prod2} we obtain
\begin{equation}
  \mathrm{det} (\alpha)=  \frac{4L_0(p+q)}{2-\alpha} = \frac{4L}{2-\alpha}\,.\label{eq:res1}
\end{equation}

\subsection{Second branch cut}
For the cut being the ray 
$\mathrm{e}^{-i\varepsilon} (0,\infty)$,
the calculation is very similar. We obtain
\begin{align}
  \zeta(s) &= 2\mathrm{e}^{-i\pi s}\cos{\frac{\pi s}{2}}\mathrm{e}^{s\log{\frac{L_0}{\pi}}}\zeta_\mathrm{R}(s) +\sum_{k=1}^{p+q-1} \left[\mathrm{e}^{-s \log{\mu_k}}+\left(\frac{\pi}{L_0}\right)^{-s}\mathrm{e}^{-\frac{i\pi}{2}s}\zeta_\mathrm{H}(s,1-i\frac{\mu_k L_0}{\pi})\right.
\\
 & \hspace*{5mm}+\left.\left(\frac{\pi}{L_0}\right)^{-s}\mathrm{e}^{-\frac{3i\pi}{2}s}\zeta_\mathrm{H}(s,1+i\frac{\mu_k L_0}{\pi})\right]\,.
\end{align}
Hence 
\begin{align*}
  \zeta'(0) &= 2 (\log{\frac{L_0}{\pi}}-i\pi)\zeta_\mathrm{R}(0)+ 2 \zeta_\mathrm{R}'(0)+\sum_{k=1}^{p+q-1}\left[-\log{\mu_k}+\left(\log{\frac{L_0}{\pi}}-\frac{i\pi}{2}\right)\zeta_\mathrm{H}(0,1-i\frac{\mu_k L_0}{\pi})\right.\\ 
& \hspace*{5mm} + \left.\left(\log{\frac{L_0}{\pi}}-\frac{3i\pi}{2}\right)\zeta_\mathrm{H}(0,1+i\frac{\mu_k L_0}{\pi})+\frac{\partial}{\partial s}\zeta_\mathrm{H}(0,1-i\frac{\mu_k L_0}{\pi})+\frac{\partial}{\partial s}\zeta_\mathrm{H}(0,1+i\frac{\mu_k L_0}{\pi})\right]\\
 & = i\pi -\log{\frac{L_0}{\pi}}-\log{(2\pi)}+\sum_{k=1}^{p+q-1}\left[-\log{\mu_k}-\log{\frac{L_0}{\pi}}-i\pi-\mu_k L_0\right.\\
 & \hspace*{5mm}+\left. \log{\left(\frac{\Gamma(1-i\frac{\mu_k L_0}{\pi})\Gamma(1+i\frac{\mu_k L_0}{\pi})}{2\pi}\right)}\right]\\
 & = i\pi(2-p-q)-\log{L_0}-(p+q)\log{2}-\sum_{k=1}^{p+q-1}[\log{(\sinh{(L_0\mu_k)})}+L_0\mu_k]\,.
\end{align*}
The determinant is equal to
\begin{align}
  \mathrm{det}(\alpha) &= \mathrm{e}^{-\zeta'(0)} \\
 & = (-2)^{p+q}L_0\prod_{k=1}^{p+q-1}\sinh{(L_0\mu_k)}\mathrm{e}^{L_0\mu_k} \\
 & = (-2)^{p+q}L_0 (-2)^{-p-q+1}\prod_{k=1}^{p+q-1}(1-\mathrm{e}^{2L_0\mu_k}) \\
 & = - 4 L_0 \frac{p+q}{\alpha+2} \\
 & = -\frac{4L}{\alpha+2}\,,\label{eq:res2}
\end{align}
where we have used \eqref{eq:prod1}. 

\subsection{Cases $\alpha=\pm 2$}
To solve the singular situations of $\alpha$ equal to $\pm 2$, let us start with the general case $p>q$ first. We will solve $p=q=1$ later.

For $\alpha=-2$ the equation \eqref{eq:z3} gives 
\begin{equation}
  z^{p-1}+z^{p-2}+\dots +z^q+2z^{q-1}+2 z^{q-2}+\dots +2 = 0\,.\label{eq:pm2_z1}
\end{equation}
Let $z_{-2,k}=\mathrm{e}^{2\mu_{-2,k} L_0}$, $k=1,\dots, p-1$ be its roots. Then similarly to the construction for general $\alpha$ we get by substituting $z=1$ that
\begin{equation}
  \prod_{k=1}^{p-1} (1-\mathrm{e}^{2\mu_{-2,k} L_0}) = p-q+2q = p+q\,.\label{eq:pm2_prod1}
\end{equation}
Dividing \eqref{eq:pm2_z1} by $2z^{p-1}$ we get the equation for $y = \mathrm{e}^{-2\lambda L_0}$
\begin{equation}
  y^{p-1}+y^{p-2}+\dots +y^{p-q}+\frac{1}{2}y^{p-q-1}+\frac{1}{2}y^{p-q-2}+\dots +\frac{1}{2} = 0\,.\label{eq:pm2_y1}
\end{equation}
For its roots $y_{-2,k}=\mathrm{e}^{-2 \mu_{-2,k} L_0}$, $k=1,\dots, p-1$ we get
\begin{equation}
  \prod_{k=1}^{p-1} (1-\mathrm{e}^{-2\mu_{-2,k} L_0}) = q+\frac{1}{2}(p-q) = \frac{1}{2}(p+q)\,.\label{eq:pm2_prod2}
\end{equation}

The construction of the spectral zeta function and the spectral determinant goes through similarly as in the general case. For the first branch cut we obtain
$$
  \zeta'(0) = -p\log{2}-\log{L_0}+\sum_{k=1}^{p-1} [L_0\mu_{-2,k}-\log{(\sinh{(L_0\mu_{-2,k})})}]
$$
and so
$$
  \mathrm{det}(-2) = 2^{p} L_0 \prod_{k=1}^{p-1}[\sinh{(L_0\mu_{-2,k})}\,\mathrm{e}^{-L_0\mu_{-2,k}}] = 2L_0 \prod_{k=1}^{p-1}(1- \mathrm{e}^{-2 L_0 \mu_{-2,k}}) = 2L_0 \frac{1}{2}(p+q) = L\,,
$$
where we have used \eqref{eq:pm2_prod2}.

The second branch cut yields
$$
  \zeta'(0) = (2-p) i\pi -p\log{2}-\log{L_0}-\sum_{k=1}^{p-1} [L_0\mu_{-2,k}+\log{(\sinh{(L_0\mu_{-2,k})})}]
$$
and so
$$
  \mathrm{det}(-2) = (-2)^{p} L_0 \prod_{k=1}^{p-1}[\sinh{(L_0\mu_{-2,k})}\,\mathrm{e}^{L_0\mu_{-2,k}}] = -2L_0 \prod_{k=1}^{p-1}(1- \mathrm{e}^{2 L_0 \mu_{-2,k}}) = -2L_0 (p+q) = -2L\,,
$$
where we have used \eqref{eq:pm2_prod1}. 

For $\alpha = 2$, we proceed similarly. We obtain the equations
\begin{align}
  z^{p-1}+z^{p-2}+\dots + z^{p-q}+\frac{1}{2}z^{p-q-1}+\frac{1}{2}z^{p-q-2}+\dots+\frac{1}{2} & = 0\,,\\
  y^{p-1}+y^{p-2}+\dots + y^{q}+2y^{q-1}+2y^{q-2}+\dots+2 & = 0 \,,
\end{align}
for variables $z = 1/y= \mathrm{e}^{2 L_0 \lambda}$. For their roots $\mathrm{e}^{2 L_0 \mu_{2,k}}$ and $\mathrm{e}^{-2 L_0 \mu_{2,k}}$, respectively, we obtain the formulae
\begin{align}
  \prod_{k=1}^{p-1} (1-\mathrm{e}^{2 L_0 \mu_{2,k}}) & = q+\frac{1}{2}(p-q) = \frac{1}{2}(p+q)\,,\\
  \prod_{k=1}^{p-1} (1-\mathrm{e}^{-2 L_0 \mu_{2,k}}) & = p-q+2q = p+q\,.
\end{align}
The determinants are for the first branch cut
$$
  \mathrm{det}(2) = 2 L_0 \prod_{k=1}^{p-1} (1-\mathrm{e}^{-2 L_0 \mu_{2,k}}) = 2L_0 (p+q) = 2L
$$
and for the second cut
$$
  \mathrm{det}(2) = - 2 L_0 \prod_{k=1}^{p-1} (1-\mathrm{e}^{2 L_0 \mu_{2,k}}) = -2L_0 \frac{1}{2}(p+q) = -L\,.
$$

Finally, we solve the case $p=q=1$ for $\alpha=\pm 2$. The equation \eqref{eq:z3} does not have a solution, hence there are only roots of the equation $\mathrm{e}^{2L_0\lambda} = 1$, $\lambda_{1j} = \frac{j\pi i}{L_0}$, $j\in\mathbb{Z}\backslash \{0\}$. This problem has been solved in \cite{FL19} and the result is 
$\mathrm{det}\,(\pm 2) = 2L_0 = L$ for the first cut and $\mathrm{det}\,(\pm 2) = -2L_0 = -L$ for the second one.

\section{General position of the Dirac damping}\label{sec:4}
In this section, we generalise the results of the previous section to irrational positions of the $\delta$-damping. We will use results on positions of the zeros of certain functions. 
The following lemma summarises the results of Theorems 12.4 and 12.5 of \cite{BC63}.
\begin{lemma}\label{lem:bc}
Let 
\begin{equation}
g(\lambda) = \sum_{j=0}^n r_j \mathrm{e}^{\beta_j \lambda}\,,\label{eq:bellman_cooke}
\end{equation} 
with $0 = \beta_0 <\beta_1<\dots <\beta_n$ and $r_j \ne 0$, $j=0,\dots, n$ are complex numbers. 
\begin{enumerate}
\item[a)] There is a positive number $c_1$ such that all zeros of $g(\lambda)$ lie in the strip $|\mathrm{Re\,}\lambda|<c_1$. 
\item[b)] Let $R$ be the rectangle $|\mathrm{Re\,}\lambda|<c_1$, $|\mathrm{Im}\,\lambda-A|\leq B$ such that no zeros of $g(\lambda)$ lie on the boundary of $R$. Then provided $c_1$ is sufficiently large, the number $n(R)$ of zeros of $g(\lambda)$ in the rectangle $R$ satisfies
$$
  -n+\frac{B}{\pi}(\beta_n-\beta_0)\leq n(R) \leq n+\frac{B}{\pi}(\beta_n-\beta_0)\,.
$$
\end{enumerate}
\end{lemma}

With the use of this lemma, we give the following theorem.

\begin{theorem}\label{thm:fj}
For a given $a\in (0,L/2]$ and $\alpha \ne \pm 2$ let $\lambda_j(a)$ be the $j$-th eigenvalue in the upper half-plane sorted in the non-decreasing order according to the imaginary part. Then $\lambda_j(a) = \frac{j\pi i}{L}+ f_j(a)$, where 
\begin{enumerate}
\item[a)] $|\mathrm{Re\,}f_j(a)|<c_1$,
\item[b)] $|\mathrm{Im\,}f_j(a)|<c_2$, where the constants $c_1$, $c_2$ are independent of $j$ and $a$, 
\item[c)] $f_j(a)$ are analytic functions in $a$ with at most algebraic singularities. If for certain $a_0$ a finite number of $\lambda_j(a)$ have the same imaginary part, one may need to interchange their indices to get the analyticity.
\end{enumerate}
A similar claim holds for the eigenvalues in the lower half-plane.
\end{theorem}
\begin{proof}
The spectral condition \eqref{eq:basic} can be rewritten as
$$
  2\sinh{(L \lambda)} + \alpha\cosh{(L\lambda)} = \alpha \cosh{((L-2a)\lambda))}
$$
or 
\begin{equation}
  (2+\alpha)\mathrm{e}^{2L\lambda}-\alpha \mathrm{e}^{(2L-2a)\lambda}-\alpha \mathrm{e}^{2a\lambda}+\alpha-2 = 0\,.\label{eq:spcond3}
\end{equation}
The condition \eqref{eq:spcond3} is of the form \eqref{eq:bellman_cooke} with $n = 3$ and $\beta_n = 2L$. Note that the leading and the last coefficients are non-zero for $\alpha \ne \pm 2$. For $\alpha = 0$ we have $n = 1$. 
\begin{enumerate}
\item[a)] The claim a) directly follows from claim a) in Lemma~\ref{lem:bc}. 
\item[b)] Let us choose the rectangle $R$ in Lemma~\ref{lem:bc} with $A = B$ (if there are eigenvalues on the positive real axis, we shift the rectangle slightly below, i.e. choose $A = B-\varepsilon$ with small $\varepsilon$). Then from claim b) in Lemma~\ref{lem:bc} it follows (for the eigenvalue near the upper side of the rectangle it holds $\mathrm{Im\,}\lambda_j \sim 2B$)
\begin{align}
  -3 + \frac{2L}{\pi} \frac{\mathrm{Im\,\lambda_j}}{2}\leq j \leq 3 + \frac{2L}{\pi} \frac{\mathrm{Im\,\lambda_j}}{2}\,,\\
  -\frac{3\pi}{L} + \frac{j\pi}{L} \leq \mathrm{Im\,\lambda_j} \leq \frac{3\pi}{L} + \frac{j\pi}{L}\,.
\end{align}
Hence the claim b) with $c_2 = \frac{3\pi}{L}$ holds.
\item[c)] The claim follows from \cite[Thm. VII 1.8]{Kat95}. The eigenvalues in dependence of the parameter $a$ are either continuous lines or two or several lines can meet at one point and emanate from this point at different angles than the incoming angles. Note that to keep the continuity of the eigenvalues, the indices of two or more eigenvalues (as defined in Thm.~\ref{thm:fj}) may be interchanged. 
\end{enumerate}
\end{proof}

We use Theorem~\ref{thm:fj} to prove the continuity of 
the determinant in the position of the Dirac damping.

\begin{theorem}\label{thm:cont1}
Let $\alpha \ne \pm 2$ and $a\in (0,L/2]$. Then the spectral determinant is continuous in $a$. 
\end{theorem}
\begin{proof}
We have 
$$
  |\lambda_j(a)^{-s}| = \left||\lambda_j|^{-\mathrm{Re}\,s}\mathrm{e}^{-i \,\mathrm{arg}\,(\lambda_j)\,\mathrm{Re}\,s}|\lambda_j|^{-i\, \mathrm{Im\,}s}\mathrm{e}^{\mathrm{arg}\,(\lambda_j) \,\mathrm{Im}\,s}\right| = |\lambda_j|^{-\mathrm{Re}\,s} |\mathrm{e}^{\mathrm{arg\,}(\lambda_j)}\,\mathrm{Im}\,s|\,.
$$
For any fixed $s$ with $\re s>1$ the sum $\sum_{j=1}^{\infty}|\lambda_j(a)^{-s}|$ is convergent. Let  $\zeta(s,a) = \sum_{j=1}^{\infty} \lambda_j(a)^{-s} $. Hence for $s$ with $\re s >1$ it holds due to the absolute summability of the sum that $\zeta(s,a_0) = \lim_{a\to a_0}\zeta(s,a)$. Both expressions can be uniquely continued to zero and hence $\zeta'(0,a_0) = \lim_{a\to a_0}\zeta'(0,a)$ and $\lim_{a\to a_0}\mathrm{det}(a) = \mathrm{det}(a_0)$.
\end{proof}

Similarly, we prove the continuity for $\alpha = \pm 2$.

\begin{theorem}
Let $\alpha = \pm 2$ and $a\in (0,L/2)$. Then the spectral determinant is continuous in $a$. 
\end{theorem}
\begin{proof}
The proof is similar to the proofs of Theorems~\ref{thm:fj} and~\ref{thm:cont1}. The condition~\eqref{eq:spcond3} becomes 
$$
  \mathrm{e}^{2(L-a)\lambda}+\mathrm{e}^{2a\lambda}-2 = 0
$$
or 
$$
  2\mathrm{e}^{2(L-a)\lambda}-\mathrm{e}^{2(L-2a)\lambda}-1=0
$$
for $\alpha=-2$ and $\alpha  = 2$, respectively. Note that the theorem does not hold for $a = \frac{L}{2}$; in that case two of the terms in the above conditions have the same $\beta_j$'s and the assumption of Lemma~\ref{lem:bc} is not satisfied.

One can use Lemma~\ref{lem:bc} to conclude in both cases that
$$
  -\frac{2\pi}{L-a}+\frac{j\pi}{L-a}\leq \mathrm{Im\,}\lambda \leq \frac{2\pi}{L-a}+\frac{j\pi}{L-a}\,.
$$
We find $\lambda_j(a) = \frac{j\pi i}{L-a}+\frac{\tilde f_j(a)}{L-a}$, and similarly to Theorem~\ref{thm:fj} we argue that $\tilde f_j(a)$ has real and imaginary part bounded by constants independent of $j$ and $a$ and it is analytic in $a$. We conclude the proof by the same argument as in the proof of Theorem~\ref{thm:cont1}.
\end{proof}

We conclude the proof of Theorem~\ref{thm:main} by observing 
that the formulae~\eqref{eq:res1} and \eqref{eq:res2} are valid also for a general value of $a$.  Together with the fact that equation~\eqref{eq:basic} is symmetric in $a$ with respect to $L/2$, this proves Theorem~\ref{thm:main}.

We see that for the first branch cut the spectral determinant is continuous at $\alpha = -2$ having the value $L$ at this point and discontinuous at $\alpha = 2$. For the second branch cut the spectral determinant is discontinuous at $\alpha = -2$ and continuous at $\alpha = 2$.

\section*{Acknowledgements}
D.K.~was supported by the EXPRO grant No.~20-17749X 
of the Czech Science Foundation. 
J.L.~was supported by the Czech Science Foundation within the project 22-18739S.

\end{document}